\documentclass[a4paper,10pt]{article}

\input prepictex
\input pictex
\input postpictex

\usepackage{amsmath, amssymb, amsthm}
\usepackage{enumerate}
\date{}

\newtheorem{defin}{Definition}

\newtheorem{thm}{Theorem}
\newtheorem{cor}{Corollary}

\newtheorem{lemma}{Lemma}
\newtheorem{example}{Example}

\title{On the dimension of the minimal vertex cover semigroup ring of an unmixed bipartite graph}
\author{C. Bertone\thanks{{\em email} cristina.bertone@unito.it} \and V. Micale\thanks{{\em email} vmicale@dmi.unict.it}}

\begin{document}
\newcommand{\trdeg}{\mathrm {trdeg}\,}
\newcommand{\rank}{\mathrm {rank}\,}
\maketitle


\begin{abstract}
In a paper in 2008, Herzog, Hibi and Ohsugi introduced and studied the semigroup ring associated to the set of minimal vertex covers of an unmixed bipartite graph. In this paper we relate the dimension of this semigroup ring to the rank of the Boolean lattice associated to the graph.
\medskip\noindent
\end{abstract}

\section{Introduction}

Let $G$ be a finite graph without loops, multiple edges and isolated vertices and let $\mathcal M(G)$ be the set of minimal vertex covers of $G$. In \cite[Section 3]{hho} the authors introduce and study the semigroup ring associated to the minimal vertex covers of an unmixed and bipartite graph $G$.

In this paper we relate the dimension of this semigroup ring to the rank of the Boolean lattice associated to $G$.

In Section 2, we recall the concept of an unmixed bipartite graph $G$ and we give some preliminaries about
the Boolean lattice associated to $G$. In particular, we cha\-rac\-te\-ri\-ze those sublattices of the Boolean lattice which are associated to $G$ (cf. Theorem \ref{1}). Then, in the particular case of bipartite graphs, we concentrate on the concept of vertex cover algebra.

In Section 3 we define the semigroup ring associated to the
minimal vertex covers of an unmixed and bipartite graph $G$ and we
prove that its dimension equals the rank of $\mathcal L_G $ plus
one (cf. Theorem \ref{th}). As a particular case of this result we
get that the dimension of the semigroup ring associated to the
minimal vertex covers of bipartite and Cohen-Macaulay graphs on
$2n$ vertices is equal to $n+1$ (cf. Corollary \ref{co}).

\section{Preliminaries}

Throughout this paper, graphs are assumed to be finite, loopless, without multiple edges and isolated vertices. We denote by $V(G)$ the set of vertices of $G$ and by $E(G)$ the set of edges of $G$.

\begin{defin}
For a graph $G$, a subset $C$ of the set of vertices $V(G)$ is called a \textbf{vertex cover} for $G$ if every edge of $E(G)$ is incident to at least one vertex from $C$.
$C$ is a \textbf{minimal} vertex cover if for any $C'\subsetneq C$, $C'$ is not a vertex cover for $G$.
\end{defin}

\noindent Let $\mathcal M(G)$ denote the set of minimal vertex covers of $G$. In general, the minimal vertex covers of a graph do not have the same cardinality.

\begin{example}
Let $G$ be the graph with $V(G)=\{1,2,3,4,5\}$ and $E(G)=\{\{1,2\},\\
\{2,3\},\{3,4\},\{1,4\},\{4,5\}\}$. Then $\mathcal M(G)=\left\{\{2,4\},\left\{1,3,5\right\}\right\}$.
\end{example}

\begin{defin}
A graph $G$ is \textbf{unmixed} if all the elements of $\mathcal M(G)$ have the same cardinality.
\end{defin}

\begin{example}\label{primoesempio}
The graph $G$ with $V(G)=\{1,2,3,4\}$ and $E(G)=\{\{1,2\},\{2,3\},\\
\{3,4\},\{1,4\}\}$ is unmixed as $\mathcal M(G)=\left\{\{2,4\},\left\{1,3\right\}\right\}$.
\end{example}

\begin{defin}
A graph $G$ is \textbf{bipartite} if its set of vertices $V(G)$ can be divided in two disjoint subsets
$U$ and $V$ such that, for all $l\in E(G)$, we have $|l\cap U|=1=|l\cap V|$.
\end{defin}

In what follows $G$ will be assumed to be bipartite and unmixed with respect to the partition $V(G)=U\cup V$ of its vertices, where $U=\{x_1,\dots,x_m\}$ and $V=\{y_1,\dots,y_n\}$.

Since $G$ is unmixed and $U$ and $V$ are both minimal vertex cover for $G$, then $n=m$.\\
 Furthermore, let $U'\subseteq U$ and $N(U')$ be the set of those vertices $y_j\in V$ for which there exist a vertex $x_i\in U'$ such that $\{x_i,y_j\}\in E(G)$; then (cf.\cite[p. 300]{hh}),
since $(U\setminus U')\cup N(U')$ is a vertex cover of $G$  for all subset $U'$ of $U$ and since $G$ is unmixed, it follows that $|U'|\le |N(U')|$ for all subset $U'$ of $U$. Thus, the marriage theorem enable us to assume that  $\{x_i,y_i\}\in E(G)$ for $i=1,\dots,n$.\\
We can also assume that each minimal vertex cover of $G$ is of the form
$$\{x_{i_1} , \dots, x_{i_s} , y_{i_{s+1}},\dots, y_{i_n}\}$$
where $\left\{i_1,\dots,i_n\right\}=[n]=\{1,\dots, n\}$.

For a minimal vertex cover $C = \{x_{i_1} , \dots, x_{i_s} , y_{i_{s+1}},\dots, y_{i_n}\}$ of $G$, we set $C' = \{x_{i_1} ,\dots , x_{i_s}\}$. Let $\mathcal L_n$ denote the Boolean lattice of all the subset of $\{x_1,\dots, x_n\}$ and let
$\mathcal L_G = \{C'\ |\ C \text{ is a minimal vertex cover of G }\}$. One easily checks this is a sublattice of $\mathcal L_n$. Since $\mathcal L_n$ is a distributive lattice, any sublattice is distributive as well.

Actually, there is a one to one correspondence between the graphs we are studying and the sublattices of $\mathcal L_n$ containing $\emptyset$ and $\{x_1,\dots,x_n\}$:

\begin{thm}\cite[Theorem 1.2]{hho}\label{1}
Let $\mathcal L$ be a subset of $\mathcal L_n$. Then there exists a (unique) unmixed
bipartite graph $G$ on $\{x_1,\dots , x_n\}\cup \{y_1,\dots , y_n\}$ such that $\mathcal L = \mathcal L_G$ if and only if $\emptyset$
and $\{x_1,\dots, x_n\}$ belong to $\mathcal L$ and $\mathcal L$ is a sublattice of $\mathcal L_n$.
\end{thm}

\subsection{Cohen-Macaulay bipartite graphs}

Let $A$ be the polynomial ring $K[z_1,\dots,z_N]$ over a field $K$. To any graph $G$ on vertex set $[N]$, let $I(G)$ be the ideal of $A$, called the \textit{edge ideal} of $G$, generated by the quadratic monomials $z_iz_j$ such that $\{i,j\}\in E(G)$.

\begin{defin}
A graph $G$ is \textbf{Cohen-Macaulay} if the quotient ring $A/I(G)$ is Cohen-Macaulay.
\end{defin}

\noindent Let, as before, $\mathcal L_n$ denote the Boolean sublattice on  $\{x_1, \dots , x_n\}$.

 \begin{defin}
 The \textbf{rank} of a sublattice $\mathcal L$ of $\mathcal L_n$, $\rank\mathcal L$, is the non-negative integer $l$ where $l+1$ is the maximal cardinality of a chain of $\mathcal L$.
 A sublattice $\mathcal L$ of $\mathcal L_n$ is called \textbf{full} if $\rank\mathcal L=n$.
 \end{defin}

\begin{thm}\cite[Theorem 2.2]{hho}\label{4}
 A subset $\mathcal L$ of $\mathcal L_n$ is a full sublattice of $\mathcal L_n$ if and only if there exists
a Cohen-Macaulay bipartite graph $G$ on $\{x_1,\dots, x_n\}\cup \{y_1, \dots , y_n\}$ with $\mathcal L = \mathcal L_G$.
 \end{thm}

\subsection{Vertex cover algebra}

Let $G$ be a bipartite and unmixed graphs on the set of vertices $\{x_1,\dots,x_n\}\cup\{y_1,\dots,y_n\}$ and with minimal vertex cover $C = \{x_{i_1} , \dots, x_{i_s} , y_{i_{s+1}},\dots, y_{i_n}\}$. It is useful  to notice that $x_i\in C$ if and only if $y_i\notin C$.

 We can identify  $C$ with the $(0,1)$-vector, $b_C\in \mathbb N^{2n}$ such that

 \[
 b_C(j)=
 \begin{cases}
 1 \text{ if }  1\leq j\leq n \text{ and } x_j\in C\\
 1 \text{ if }  n+1\leq j\leq 2n \text{ and } y_{j-n}\in C\\
 0 \text{ otherwise}
 \end{cases}
 \]
where $b_C(j)$ denotes the $j$-th coordinate of the vector $b_C$.

 In this way, we can associate to each minimal vertex cover $C$ of $G$ a squarefree monomial in the polynomial ring $S=K[x_1,\dots,x_n,y_1,\dots,y_n]$ with $\deg x_i=\deg y_i=1$; in fact, we first associate to $C$ its vector $b_C$ and then we consider the monomial $u_{C}=x_1^{b_C(1)}\cdots x_n^{b_C(n)}y_1^{b_C(n+1)}\cdots y_n^{b_C(2n)}$.

\begin{defin} The \textbf{vertex cover algebra} of the bipartite graph $G$ is the subalgebra $A(G)$ of $S[t]$ generated, over $S$, by the monomials $u_{C}t$ for every minimal vertex cover $C$ of $G$, that is $A(G)=S[u_{C}t,\ C\in\mathcal M(G)]$.
\end{defin}

\noindent By \cite[Theorem 4.2 and Corollary 4.4]{hht}, we have, in particular, that $A(G)$ is a finitely generated, graded, normal, Gorenstein $S$-algebra.

Moreover, in \cite[Theorem 5.1]{hht}, the authors show that $A(G)$
is generated in degree $\leq 2$ and that it is standard graded.

 \section{The dimension of $\overline{A(G)}$}

 We now introduce the object of our study in this paper.

Let $\mathfrak m$ be the maximal graded ideal of $S$. For a bipartite unmixed graph $G$, we consider the standard graded $K$-algebra
 \[
 \overline{A(G)}:=A(G)/\mathfrak m A(G)\cong K[u_{C}t,\ C\in\mathcal M(G)]\cong K[u_{C},\ C\in\mathcal M(G)].
 \]

 \noindent Hence $ \overline{A(G)}$ is the semigroup ring generated by all monomials $u_{C}$ such that $C\in\mathcal M(G)$. This object has been introduced and studied in \cite[Section 3]{hho}, where the authors proved, in particular, that $\overline{A(G)}$ is a normal and Koszul semigroup ring (cf. \cite[Corollary 3.2]{hho}).

The aim of the paper is to relate the dimension of $\overline{A(G)}$ to $\rank\mathcal L_G$ (cf. Theorem \ref {th}).

Let $d=|\mathcal M(G)|$ and $B_G$ be the $d\times 2n$ matrix whose
rows are exactly the vectors $b_C$.
 Since $C_1=\{x_1,\dots,x_n\}$ and $C_2=\{y_1,\dots,y_n\}$ are always in $\mathcal M(G)$, we can assume that the first and the last rows
 of $B_G$ are $b_{C_1}$ and $b_{C_2}$ respectively. Finally, let $\widetilde{b_C}$ be the $n$-vector containing only the first $n$ entries of $b_C$ and let $\widetilde{B_G}$ be the $d\times n$ matrix whose rows are the vectors $\widetilde{b_C}$.

 \begin{lemma}\label{2matrici}
 $\rank B_G=\rank \widetilde{B_G}+1$.
 \end{lemma}
 \begin{proof}
  Let $C_1,\dots,C_{2n}$ be the  column vectors of $B_G$ (note that the columns of $\widetilde{B_G}$ are exactly $C_1,\dots, C_n$) and
 let $\widetilde C$ be the column vector with $d$ entries each equals to 1.
 Since $C_{n+j}=\widetilde C-C_j$ for every $j=1,\dots,n$, then we have that
 \[
 \langle C_1,\dots,C_n,\widetilde C\rangle_K=\langle C_1,\dots,C_n,C_{n+1},\dots,C_n\rangle_K.
 \]
as $K$-vector spaces.

Finally, since the last entry in each column $C_1,\dots,C_n$ is $0$, it follows that $\widetilde C\notin\langle C_1,\dots,C_n\rangle_K$, that is
 \[
 \dim_K\langle C_1,\dots,C_n,\widetilde C\rangle_K=\dim\,\langle C_1,\dots,C_n\rangle_K+1.
 \]
  \end{proof}

 \begin{lemma}\label{reticolomatrice}
  $\rank \widetilde{B_G}=\rank \mathcal L_G$
  \end{lemma}
  \begin{proof}
 Let $\rank\mathcal L_G=m$ and consider a chain of maximal length $m+1$ in the sublattice $\mathcal L_G$. We note that, by Theorem \ref{1}, $\emptyset$ and $[n]$ are in this chain. Each element of this maximal chain corresponds to a row of the matrix $\widetilde{B_G}$. Let denote with $v_1,\dots, v_{m+1}$ the row vectors associated to this maximal chain, where $v_1$ is the vector associated to the element at the top of the chain, $v_2$ is the vector associated to the element of the chain just below the top, and so on for the remaining vectors $v_3,\dots,v_{m+1}$. With this notation we have that $v_1$ is the vector with all $1$'s and $v_{m+1}$ is the vector with all $0$'s.
We note that if $i>j$, then the numbers of $1$'s in $v_i$ is strictly less than the number of $1$'s in $v_j$ and that if $0$ is the $l$-th coordinate of $v_i$, then $0$ is the $l$-th coordinate of $v_j$. This two facts imply that $v_1,\dots, v_{m}$ are linearly indipendent. So we have $\rank\widetilde{B_G}\geq m$.

 In order to prove that equality holds, we show that all the other rows of $\widetilde{B_G}$ are linear combination of the $m$ rows associated to $v_1,\dots, v_m$. With an abuse of notation, we now identify the elements of the lattice $\mathcal L_G$ with their associated vectors.
 Since $\mathcal L_G$ is a lattice containing $[n]$ and $\emptyset$, following the maximal chain in the lattice containing the vectors $v_1,\dots,v_m$, we have, at a certain height, the situation depicted in the picture

 
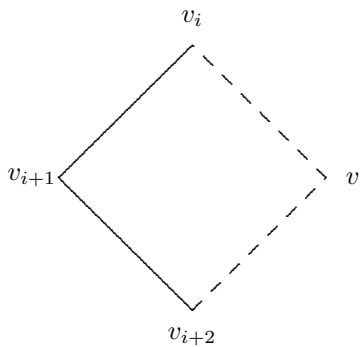
\vspace{.5 cm}

 $$  \beginpicture
  \setcoordinatesystem units <2pt,2pt>
\setplotarea x from 0 to 50, y from 0 to 50
\setlinear
\plot 25 50  0 25 /
\plot 0 25  25 0 /
\setdashes
\plot 25 0  50 25 /
\plot 50 25 25 50 /
   \put{$v_i$} at 25 55
  \put{$v_{i+1}$} at -5 25
    \put{$v_{i+2}$} at 25  -5
   \put{$v$} at 55 25
\endpicture
$$

\vspace{.5cm}

\noindent where $v_i, v_{i+1}, v_{i+2}$ are in the maximal chain.

But $\mathcal L_G$ is a distributive lattice and, in terms of the vectors, this means that we can obtain $v$ from the other vectors in the picture: in fact (vectorially)
\[
v=v_{i}-v_{i+1}+v_{i+2}.
\]

Repeating this in each analogous situation, we have that all the
possible vectors representing elements of the lattice which are not
in the chosen maximal chain, can be obtained by a linear
combination of the vectors $v_1,\dots, v_m$. In terms of the
matrix $\widetilde{B_G}$, this means that $\rank\widetilde{B_G}\le
m$.

 \end{proof}

 \begin{thm}\label{th}
 Let $G$ be an unmixed, bipartite graph on $2n$ vertices with no isolated vertices and let $\mathcal L_G$ be the associated sublattice of $\mathcal L_{n}$. Then
 \[
 \dim\, \overline{A(G)}=\rank \mathcal L_G+1.
 \]
  \end{thm}

 \begin{proof}
By  \cite[Proposition 7.1.17]{v1}, we have that $\dim\, \overline{A(G)}=\rank\, B_G$. By Lemmas \ref{2matrici} and \ref{reticolomatrice}, we get the proof.
 \end{proof}

\begin{cor}\label{co}
Let $G$ be a Cohen-Macaulay bipartite graph on $2n$ vertices. Then
\[
\dim\, \overline{A(G)}=n+1
\]
\end{cor}

\begin{proof}
By \cite[Proposition 6.1.21]{v1}, $G$ is unmixed. Furthermore, by
Theorem \ref{4}, Cohen-Macaulay graphs correspond to full
sublattices. Hence, by Theorem \ref{th}, we get the thesis.
\end{proof}

\bigskip
         {\large
  {\bf Acknowledgements}}

\medskip

The authors wish to thank Professors J\"urgen Herzog and Volkmar
Welker for the valuable conversations concerning this paper. We
also wish to thanks Professor Alfio Ragusa and all other
organizers of PRAGMATIC 2008 for the opportunity to partecipate and for the pleasant atmosphere they
provided during the summer school.

\end{document}